\documentclass[a4paper]{article}
\usepackage{graphicx} % Required for inserting images
\usepackage{amssymb,amsmath,enumitem}
\usepackage{amsthm}
\usepackage{thmtools,thm-restate}

\usepackage{authblk}
\usepackage{lineno}
\usepackage{xcolor}
\usepackage{tikz}
\usepackage[hyphens]{xurl}
\usepackage[colorlinks=true, linkcolor = violet, citecolor = blue, urlcolor = blue]{hyperref}
\hypersetup{
	pdftitle = {Linear Bounds on Treewidth in terms of Excluded Planar Minors},
	pdfauthor = {J. Pascal Gollin, Kevin Hendrey, Sang-il Oum, and Bruce Reed}
}
\declaretheorem[name=Lemma, numberwithin = section]{lemma}
\declaretheorem[name=Theorem,sibling = lemma]{theorem}
\declaretheorem[name=Proposition, sibling=lemma]{proposition}

\declaretheorem[name=Corollary, sibling=lemma]{corollary}
\declaretheorem[name=Conjecture, sibling=lemma]{conjecture}
\declaretheorem[name=Question, sibling=lemma]{question}
\declaretheorem[name=Claim]{claim}

\usepackage[noabbrev,capitalise,nameinlink]{cleveref}
\crefname{claim}{Claim}{Claims}
\crefname{lemma}{Lemma}{Lemmas}
\crefname{theorem}{Theorem}{Theorems}
\crefname{proposition}{Proposition}{Propositions}
\crefname{question}{Question}{Questions}
\crefname{conjecture}{Conjecture}{Conjectures}
\crefname{figure}{Figure}{Figures}
\crefname{corollary}{Corollary}{Corollaries} 
\crefformat{equation}{(#2#1#3)}
\Crefformat{equation}{Equation #2(#1)#3}
\newenvironment{subproof}[1][\proofname]{%
  \begin{proof}[#1]%
}{%
  \end{proof}%
}

\newcommand\abs[1]{\lvert#1\rvert}

\renewcommand{\epsilon}{\varepsilon}  
      
\newcommand{\Beta}{\mathcal B}

\DeclareMathOperator{\ord}{ord}
\DeclareMathOperator{\tw}{tw}

\title{Linear bounds on treewidth \\
in terms of excluded planar minors} 
\author{J. Pascal Gollin\thanks{Supported by the Institute for Basic Science (IBS-R029-Y3).}}
\author{Kevin Hendrey\thanks{Supported by the Institute for Basic Science (IBS-R029-C1).}}
\author[$\dagger$1,2]{Sang-il Oum}
\affil[1]{Discrete Mathematics Group,
Institute for Basic Science (IBS),
Daejeon,~South~Korea}
\affil[2]{Department of Mathematical Sciences, KAIST, Daejeon,~South~Korea}
\author[3]{Bruce~Reed\thanks{Supported by  NSTC Grant 112-2115-M-001 -013 -MY3}}
\affil[3]{Mathematical Institute, Academia Sinica, Taiwan.} 
\affil[ ]{Email: 
\texttt{pascalgollin@ibs.re.kr},
\texttt{kevinhendrey@ibs.re.kr},
\texttt{sangil@ibs.re.kr}, \texttt{bruce.al.reed@gmail.com}}
\date{February 27, 2024}

\begin{document}
%\linenumbers
\maketitle

\begin{abstract} 
    One of the fundamental results in graph minor theory is that for every planar graph~$H$,
    there is a minimum integer~$f(H)$ such that graphs with no minor isomorphic to~$H$ have treewidth at most~$f(H)$. 
    A lower bound for~${f(H)}$ can be obtained by considering the maximum integer~$k$ such that~$H$ contains~$k$ vertex-disjoint cycles. 
    There exists a graph of treewidth~${\Omega(k\log k)}$ which does not contain~$k$ vertex-disjoint cycles, from which it follows that~${f(H) = \Omega(k\log k)}$. 
    In particular, if~${f(H)}$ is linear in~${\abs{V(H)}}$ for graphs~$H$ from a subclass of planar graphs, it is necessary that $n$-vertex graphs from the class contain at most~${O(n/\log(n))}$ vertex-disjoint cycles. 
    We ask whether this is also a sufficient condition, and demonstrate that this is true for classes of planar graphs with bounded component size. 
    For an $n$-vertex graph~$H$ which is a disjoint union of~$r$ cycles, we show that~${f(H) \leq 3n/2 + O(r^2 \log r)}$, and improve this to~${f(H) \leq n + O(\sqrt{n})}$ when~${r = 2}$. 
    In particular this bound is linear when~${r=O(\sqrt{n}/\log(n))}$. 
    We present a linear bound for~${f(H)}$ when~$H$ is a subdivision of an $r$-edge planar graph for any constant~$r$.
    We also improve the best known bounds for~${f(H)}$ when~$H$ is the wheel graph or the ${4 \times 4}$ grid, obtaining a bound of~$160$ for the latter.
\end{abstract}

 \section{Introduction}

A \emph{tree decomposition} of a graph~$G$ consists of a tree~$T$ and a subtree~$S_v$ of~$T$ for each vertex~$v$ of~$G$.
For each node~$t$ of the tree~$T$, we let~${X_t = \{v \mid t \in V(S_v)\}}$ 
and define the \emph{width} of the tree decomposition as the maximum of~${\abs{X_t}-1}$ over the nodes~$t$ of the tree. 
The \emph{treewidth} of~$G$, denoted by~$\tw(G)$, is the minimum of the widths of its tree decompositions. 

One of the fundamental results in graph minor theory, proved by Robertson and Seymour~\cite{RS1991}, is that for every planar graph~$H$,
there is a minimum integer~$f(H)$ such that graphs which do not contain~$H$ as a minor have treewidth at most~$f(H)$. 
When~$H$ is a ${k \times k}$ grid, the best known upper bound on~$f(H)$, obtained by Chuzhoy and Tan~\cite{CT2020}, is~${O(k^9 \operatorname{poly~log} k)}$. 
This implies that~${f(H) = O(\abs{V(H)}^9 \operatorname{poly~log} \abs{V(H)})}$ for arbitrary planar graphs~$H$, as Robertson, Seymour, and Thomas~\cite{RST1994} proved that every planar~$H$ is a minor of 
a ${k \times k}$ grid for~${k = O(\abs{V(H)})}$. 

It is natural to ask for a better bound on~$f(H)$ for~$H$ in various classes of planar graphs. 
Any bound must be~${\Omega(\abs{V(H)})}$ as the complete graph on~${\abs{V(H)}-1}$ vertices has treewidth~${\abs{V(H)}-2}$ and does not contain~$H$ as a minor. 
This paper focuses on~$H$ for which~$f(H)$ is~${O(\abs{V(H)})}$. 

Several authors have presented results showing that~${f(H) = O(\abs{V(H)})}$ for various special~$H$. 
Bienstock, Robertson, Seymour, and Thomas~\cite{BRST1991} showed that when~$H$ is a forest, $f(H)$ is~${\abs{V(H)}-2}$. 
Fellows and Langston~\cite{FL1989} showed that if~$H$ is a cycle, then~$f(H)$ is again~${\abs{V(H)}-2}$. 
Bodlaender, van Leeuwen,  Tan, and Thilikos~\cite{BLTT1997} showed that~${f(K_{2,t}) \leq 2t-2}$ for every integer~${t \geq 2}$. 
Raymond and Thilikos~\cite[Theorem 5.1]{RT2017a} proved that 
\begin{equation} 
    f(H) \leq 36 \abs{V(H)} - 39 \label{eq:wheel}
\end{equation}
for every wheel graph $H$. 
Leaf and Seymour~\cite[4.4]{LS2015} proved that for an apex forest $H$ with at least two vertices, which is a graph that becomes a forest by deleting some vertex, ${f(H) \leq \frac{3}{2} \abs{V(H)} - 3}$.
Liu and Yoo~\cite{LY2024} informed us that, in a manuscript under preparation, they proved $f(H)\leq \abs{V(H)}-2$ for an apex forest~$H$, improving the bound of Leaf and Seymour. 

Not every planar graph~$H$ has the property that~${f(H) = O(\abs{V(H)})}$. 
Robertson, Seymour, and Thomas~\cite{RST1994} pointed out that~${f(H) = \Omega(g^2 \log g)}$ for the ${g \times g}$ grid~$H$. 
For this, they use a probabilistic argument to show that for some fixed positive~$\epsilon$ and every sufficiently large integer~$n$, 
there are $n$-vertex graphs~$G$ with treewidth exceeding~$\epsilon n$ and girth at least~${\epsilon \log n}$. 
Thus for large~$g$, if we choose~${n = \lceil \frac{1}{9}\epsilon g^2 \log g \rceil}$, then~$G$ has no~${g \times g}$ grid~$H$ as a minor because $H$ contains~${\lfloor g^2/9\rfloor }$ vertex-disjoint cycles
so~${f(H) \geq \tw(G) \geq \epsilon n \geq \frac{1}{9}\epsilon^2 g^2 \log g}$. 

As was implicitly pointed out by Cames van Batenburg, Huynh, Joret, and Raymond~\cite{CHJR2019}, 
by the same method, we deduce the following generalization. 
\begin{proposition}
    \label{prop:fewcycles}
    For every~$c>0$, there is~$d>0$ such that for every graph $H$ with at least two vertices, if~${f(H) \leq c\abs{V(H)}}$, then~$H$ contains at most~${\frac{d\abs{V(H)}}{\log \abs{V(H)}}}$ vertex-disjoint cycles. 
\end{proposition}

\begin{proof}
    There is~${\epsilon \in (0,1)}$ and an integer~${n_0 > 1}$ such that for all integers~${n \geq n_0}$, there is an $n$-vertex graph~$G$ with~${\tw(G) > \epsilon n}$ and girth at least~${\epsilon \log n}$.
    We may assume that~${c > 1}$. 
    We set~${d := \max(\frac32 c\epsilon^{-2}, \lceil \log n_0 \rceil)}$. 
    Since~$H$ has at most $\abs{V(H)}$ vertex-disjoint cycles, 
    we may assume that~${\abs{V(H)} \geq n_0}$. 

    Let~${n := \lceil c\epsilon^{-1}\abs{V(H)} \rceil\ge n_0}$. 
    Let $G$ be an $n$-vertex graph which has treewidth more than~${\varepsilon n \geq c\abs{V(H)}}$ and girth at least~${\epsilon \log n}$. 
    Then~$G$ has at most~${n/(\epsilon \log n)}$ vertex-disjoint cycles.
    Since~$H$ is a minor of~$G$, the maximum number of vertex-disjoint cycles in~$H$ is at most
    \[
        \frac{n}{\epsilon \log n} 
        \leq \frac{\frac32 c\varepsilon^{-1} \abs{V(H)} }{\epsilon\log n}
        \leq \frac{d\abs{V(H)}}{\log \abs{V(H)}}.\qedhere
    \]     
\end{proof}

Our first theorem is a partial converse to this result. 

\begin{restatable}{theorem}{smallcomponents}
    \label{smallcomponents}
    Let~$r$ be a fixed positive integer and~$H$ be a planar graph with at least two vertices. 
    If every component of~$H$ is a tree or has at most~$r$ vertices, then~${f(H) = O(\abs{V(H)})}$ 
    precisely if~$H$ has at most~${O(\frac{\abs{V(H)}}{\log \abs{V(H)}})}$ components having cycles. 
\end{restatable} 

Our second result shows that~$f(H)$ is~$O(\abs{V(H)})$ whenever~$G$ is the disjoint union of~${O(\frac{\sqrt{\abs{V(H)}}}{\log \abs{V(H)}})}$ disjoint cycles: 

\begin{restatable}{theorem}{fewcycles}
    \label{fewcycles}
    There is an absolute constant~$c$ such that for every~${r \geq 3}$, 
    if~$H$ is the disjoint union of~$r$ cycles, then 
    \[
        {f(H) \leq \frac{3\abs{V(H)}}{2} + cr^2 \log r}.
    \]
    If~$H$ is the disjoint union of two cycles, then 
    \[
        {f(H) < \abs{V(H)} + \frac{9}{2} \left\lceil \sqrt{4 + \abs{V(H)}} \right\rceil + 2}.
    \]
    % \[ f(H) \leq \abs{V(H)}+10\sqrt{\abs{V(H)}}+10.\]
\end{restatable}

Our third result shows that~$f(H)$ is~$O(\abs{V(H)})$ whenever~$G$ is the subdivision of a planar graph with~$O(1)$ edges. 

\begin{restatable}{theorem}{outerplanar}
    \label{outerplanar}
    For every integer~${r \geq 2}$, there is a constant~$b_r$ such that if~$H$ is a subdivision of a planar graph with at most~$r$ edges, then~${f(H) \leq \frac{r+1}{2} \abs{V(H)} + b_r}$. 
\end{restatable}

\begin{figure}
    \centering
    \begin{tikzpicture}
        \tikzstyle{v}=[circle, draw, solid, fill=black, inner sep=0pt, minimum width=3pt]
        % Draw the vertices
        \foreach \i in {1,...,5} {
            \node[v] (a\i) at (0,0.5*\i) {};
            \node[v] (b\i) at (2,0.5*\i) {};
            \draw (a\i) to (b\i);
        }
        % Draw a closed curve passing through all the vertices
        \draw (a1) to (a5);
        \draw [out=180,in=180] (a1) to (a5);
        \draw (b1) to (b5);
        \draw [out=0,in=0] (b1) to (b5);
    \end{tikzpicture}
    \quad
    \begin{tikzpicture}
        \tikzstyle{v}=[circle, fill, inner sep=0pt, minimum size=3pt]
        % Draw the vertices
        \foreach \i in {1,...,8} {
            \node[v] (a\i) at (0,0.3*\i) {};
            \node[v] (b\i) at (2.5,0.3*\i) {};
        }
        % Draw a closed curve passing through all the vertices
        \draw (a1) to (a8);
        \draw [out=180,in=180] (a1) to (a8);
        \draw (b1) to (b8);
        \draw [out=0,in=0] (b1) to (b8);
        \draw (a1) to (b3);
        \draw (a2) to (b5);
        \draw (a3) to (b2);
        \draw (a4) to (b1);
        \draw (a5) to (b8);
        \draw (a6) to (b7);
        \draw (a7) to (b4);
        \draw (a8) to (b6);
    \end{tikzpicture}
    % \quad
    % \begin{tikzpicture}
    %     \tikzstyle{v}=[circle, fill, inner sep=0pt, minimum size=3pt]
    %     % Draw the vertices
    %     \foreach \i in {1,...,8} {
    %         \node[v] (a\i) at (0,0.3*\i) {};
    %         \node[v] (b\i) at (2.5,0.3*\i) {};
    %     }
    %     % Draw a closed curve passing through all the vertices
    %     \draw (a1) to (a8);
    %     % \draw [out=180,in=180] (a1) to (a8);
    %     \draw (b1) to (b8);
    %     % \draw [out=0,in=0] (b1) to (b8);
    %     \draw (a1) to (b3);
    %     \draw (a2) to (b5);
    %     \draw (a3) to (b2);
    %     \draw (a4) to (b1);
    %     \draw (a5) to (b8);
    %     \draw (a6) to (b7);
    %     \draw (a7) to (b4);
    %     \draw (a8) to (b6);
    % \end{tikzpicture}
    \caption{The $5$-prism and an instance of a twisted $8$-prism.}
    \label{fig:twistedprism}
\end{figure}
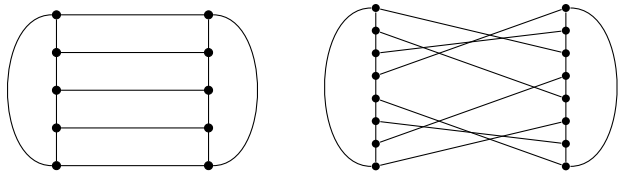

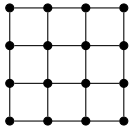
\begin{figure}
    \centering
    \begin{tikzpicture}
        \tikzstyle{v}=[circle, draw, solid, fill=black, inner sep=0pt, minimum width=3pt]
        % Draw the 4x4 grid

        \foreach \i in {1,...,4} {
            \foreach \j in {1,...,4} {
                \node[v] (a\i\j) at (0.5*\i,0.5*\j) {};
            }
            \draw (a\i1) to (a\i4);
        }
        \foreach \j in {1,...,4} {
            \draw (a1\j) to (a4\j);
        }
    \end{tikzpicture}    
    \caption{The ${4 \times 4}$ grid.}
    \label{fig:4x4grid}
\end{figure}

The \emph{$\ell$-prism} is a graph that is the Cartesian product of~$K_2$ and the cycle of length~$\ell$. 
A twisted $\ell$-prism is a graph that consists of two vertex-disjoint cycles of length $\ell$ joined by a matching of size~$\ell$. 
See \cref{fig:twistedprism} for an illustration of an $\ell$-prism and a twisted $\ell$-prism. 
Birmel\'e, Bondy, and Reed~\cite{BBR2007a}  showed that~${f(H) \leq 60\ell^2-120\ell+62}$ for the $\ell$-prism~$H$. 
They used this to show that the treewidth of any graph without a ${4 \times 4}$ grid minor was at most~$7262$. 
See \cref{fig:4x4grid} for an illustration of the ${4 \times 4}$ grid. 

Their approach was to show that if a graph~$G$ does not contain a minor isomorphic to a twisted $\ell$-prism~$H$, then its treewidth is at most~${60 \ell - 58}$. 
They then combine this with a well-known theorem of Erd\H{o}s and Szekeres \cite{ES1935}, which immediately implies that a graph with a twisted ${((\ell-1)^2+1)}$-prism minor contains an $\ell$-prism minor. 

We prove the following theorem, which is tight up to a~$o(1)$ factor. 

\begin{restatable}{theorem}{twistedprismone}
    \label{twistedprism1}
        Every graph without a twisted $\ell$-prism as a minor has treewidth at most~${2 \ell + 18 \lceil \frac{1 + \sqrt{2 \ell + 1}}{4} \rceil - 8}$. 
\end{restatable}

Since the wheel graph on~${\ell+1}$ vertices is a minor of a twisted $\ell$-prism, we deduce the following corollary, which improves the bound~\eqref{eq:wheel} for the wheel by Raymond and Thilikos~\cite[Theorem 5.1]{RT2017a}. 
\begin{corollary}
    \label{ourwheel}
    Every graph without the $k$-vertex wheel graph  as a minor has treewidth at most ${2 k + 18 \lceil \frac{1 + \sqrt{2k -1}}{4} \rceil - 10}$. \qed
\end{corollary}
We also obtain the following. 

\begin{restatable}{theorem}{twistedprismtwo}
    \label{twistedprism2}
        Every graph without a twisted $\ell$-prism or a ${4 \times 4}$ grid as a minor has treewidth at most~${2 \ell + 10}$. 
\end{restatable}

We use the latter result to show a new upper bound on the treewidth of graphs without a ${4 \times 4}$ grid minor, 
improving the previous bound $7262$ by 
Birmel\'e, Bondy, and Reed~\cite{BBR2007a}.

\begin{restatable}{theorem}{nofourbyfours}
    \label{no4by4s}
    Every graph without a ${4 \times 4}$ grid minor has treewidth at most~$160$. 
\end{restatable}

This paper is organized as follows. 
\cref{sec:prelim} recalls brambles.
In \cref{sec:smallcomponents}, we consider graphs $H$ having a few non-tree components where each non-tree component has at most~$r$ vertices, and prove \cref{smallcomponents}.
In \cref{sec:twistedprism}, we consider twisted $\ell$-prisms and prove \cref{twistedprism1,twistedprism2,no4by4s}.
In \cref{sec:fewcycles}, we consider graphs $H$ that are the disjoint union of a few cycles and prove \cref{fewcycles}.
In \cref{sec:outerplanar}, we consider graphs $H$ that are subdivisions of planar graphs with at most~$r$ edges, and prove \cref{outerplanar}.
We conclude this paper by presenting a few open problems in \cref{sec:open}.

\section{Preliminaries}\label{sec:prelim}

In proving these results, we often focus on the dual of treewidth, the \emph{bramble number}. 
A \emph{bramble}~$\Beta$ in a graph $G$ is a set of sets~${B \subseteq V(G)}$ that induce a connected subgraph~$G[B]$ such that for every two~${B, B' \in \Beta}$, we have that~${G[B \cup B']}$ is connected. 
A \emph{hitting set} for a bramble is a set of vertices intersecting all of its elements. 
The \emph{order} of a bramble~$\Beta$, denoted by~$\ord(\Beta)$, is the minimum size of a hitting set for~$\Beta$. 
The \emph{bramble number of~$G$} is the largest order of a bramble in~$G$. 
Any subset of a bramble $\Beta$ yields a new bramble, which is called a \emph{subbramble} of~$\Beta$. 
Seymour and Thomas~\cite{ST1993} showed the following duality theorem, see \cite{Reed1997}. 

\begin{theorem}[Seymour and Thomas~\cite{ST1993}]\label{duality}
    The treewidth of a graph is exactly one less than its bramble number. 
\end{theorem}

We will use the following result from Birmel\'e, Bondy, and Reed~\cite{BBR2007a}.
\begin{lemma}[Birmel\'e, Bondy, and Reed~{\cite[Theorem 2.4]{BBR2007a}}]
    \label{cyclebramble}
    Let~$G$ be a graph having a bramble~$\Beta$ of order at least three. 
    Then, there is a cycle~$C$ meeting every element of~$\Beta$. 
\end{lemma}

\section{%
\texorpdfstring{Excluding $H$ with Bounded Size Components}%
{Excluding H with Bounded Size Components}}\label{sec:smallcomponents}

In this short section, we prove \cref{smallcomponents}.
As we discussed in \cref{prop:fewcycles}, in order to have a linear bound on~$f(H)$, $H$ cannot have too many vertex-disjoint cycles. 
So if we limit our attention to graphs with at most~$r$ vertices in each non-tree component, then the number of non-tree components should be small to have a linear bound on~$f(H)$. 

To state the next lemma, we first define the pathwidth of graphs. 
A \emph{path decomposition} is a tree decomposition in which the underlying tree is a path. 
The \emph{pathwidth} of a graph~$G$ is the minimum width of its path decompositions. 
Clearly, the pathwidth of a graph is always greater than or equal to its treewidth. 

To deal with the tree components, we use the following lemma of Diestel~\cite{Diestel1995}, which was used in his short proof of the theorem of Bienstock, Robertson, Seymour, and Thomas~\cite{BRST1991} that every graph of pathwidth at least~${t-1}$ contains every tree on~$t$ vertices as a minor. 
A better presentation of its proof can be found in the proof of Theorem 12.4.5 in the first edition of the book by Diestel~\cite{Diestel1997}. 
Seymour~\cite{Seymour2023} wrote this more explicitly. 

\begin{lemma}[Diestel~\cite{Diestel1995,Diestel1997}]
    \label{diestel}
    Let~$T$ be a tree with~$t$ vertices and let~$G$ be a graph of pathwidth at least~${t-1}$. 
    Then there is a separation~${(A,B)}$ of~$G$ such that 
    \begin{enumerate}[label=(P\arabic*)]        
        \item \label{p1} ${\abs{A \cap B} = t}$, 
        \item \label{p2} ${G[A]}$ contains~$T$ as a minor where each vertex of~${A \cap B}$ appears as a distinct vertex of~$T$, and 
        \item \label{p3} ${G[A]}$ has a path decomposition of width~${t-1}$ with~${A \cap B}$ as the last bag. 
    \end{enumerate}    
\end{lemma}

By using \cref{diestel}, we deduce the following.

\begin{lemma}
    \label{removingtree}
   For any tree component~$T$ of~$H$, ${f(H) \leq f(H-V(T)) + \abs{V(T)}}$.  
\end{lemma}

\begin{proof}
    Let~$G$ be a graph that does not contain~$H$ as a minor and let~$T$ be a tree that is a component of~$H$. 
    If the pathwidth of~$G$ is less than~${\abs{V(T)} - 1}$, then so is its treewidth, and we are done. 
    Otherwise, we apply \cref{diestel} to obtain a separation~${(A,B)}$ of~$G$ satisfying \ref{p1}, \ref{p2}, and \ref{p3}. 
    We know~${G-A}$ does not contain~${H-V(T)}$ as a minor and therefore the treewidth of~${G-A}$ is at most~${f(H-V(T))}$. 
    We take a tree decomposition of~${G-A}$ having width at most~${f(H-V(T))}$, add~${A \cap B}$ to every bag, and combine it with the path decomposition of~${G[A]}$ by adding an edge from the endpoint~$x$ of the path decomposition to a node of the tree. 
    Thus we obtain a tree decomposition of~$G$ of width at most~${f(H-V(T)) + \abs{V(T)}}$. 
\end{proof}

So, we need only show that for any~$s$ and~$d$, there is an integer~$c$ such that if every component of~$H$ has size at most~$s$ and 
there are at most~$\frac{d\abs{V(H)}}{\log \abs{V(H)}}$ non-tree components then~${f(H) \leq c \abs{V(H)}}$. 
We will do so, with~$c$ defined implicitly, momentarily. 

So, we need only consider~$H$ that are the disjoint union of bounded size components, all of which contain a cycle. 
We apply the following lemma \cite[Corollary 2.2]{CHJR2019}. 

\begin{lemma}[{Cames van Batenburg, Huynh, Joret, and Raymond~\cite[Corollary 2.2]{CHJR2019}}]
    \label{eplem}
    For every integer~$r$, there is an integer~$s$ such that for every integer~$k$, every graph of treewidth at least~${s k \log (k+1)}$ contains~$k$ vertex-disjoint subgraphs of treewidth at least~$r$. 
\end{lemma}

We now present the proof of \cref{smallcomponents}. 

\smallcomponents*

\begin{proof}%[Proof of \cref{smallcomponents}]
    By \cref{prop:fewcycles}, it is enough to prove the direction that if~$H$ has at most~${O(\frac{\abs{V(H)}}{\log \abs{V(H)}})}$ components having cycles, then~${f(H) = O(\abs{V(H)})}$. 
    
    Let~${R := \max\{ f(G) \mid G \textrm{ is a planar graph with $r$ vertices} \}}$. 
    Let~$H'$ be the induced subgraph of~$H$ consisting of all components of~$H$ which are not trees. 
    Let~$k$ be the number of components of~$H'$. 
    By \cref{eplem}, there is an integer~$s$ depending only on~$R$ such that every graph of treewidth at least~${s k \log (k+1)}$ contains~$k$ vertex-disjoint subgraphs of treewidth at least~$R$. 
    By the choice of~$R$, we deduce that~${f(H') < s k \log (k+1)}$. 
    Since~${k = O(\abs{V(H)}/\log\abs{V(H)})}$, we have~${f(H') = O(\abs{V(H)})}$. 
    By applying \cref{removingtree} to each tree component of~$H$, we deduce that~${f(H) \leq f(H') + \abs{V(H)-V(H')} = O(\abs{V(H)})}$. 
\end{proof}

\section{%
\texorpdfstring{Excluding a Twisted $\ell$-Prism}%
{Excluding a Twisted l-Prism}}\label{sec:twistedprism}

In this section, we prove \cref{twistedprism1,twistedprism2,no4by4s}. 
Our approach is to try to find two vertex-disjoint cycles to which we can build a twisted $\ell$-prism. 

Here is a lemma implicitly used by Birmel\'e, Bondy, and Reed~\cite{BBR2007a}. 

\begin{lemma}
    \label{menger}
    Let~$G$ be a graph with a bramble~$\Beta$. 
    Let~$S$ and~$T$ be two vertex-disjoint subgraphs 
    such that for each of~$V(S)$ and~$V(T)$ there is a subbramble of order at least~$\ell$ of~$\mathcal{B}$ each member of which intersects~$V(S)$ or~$V(T)$, respectively. 
    Then there are~$\ell$ vertex-disjoint paths from~$V(S)$ to~$V(T)$. 
    In particular, if both~$S$ and~$T$ are cycles, then~$G$ has a twisted $\ell$-prism as a minor. 
\end{lemma}

\begin{proof}
    Suppose for a contradiction that there are no~$\ell$ vertex-disjoint paths between~$V(S)$ and~$V(T)$. 
    By Menger's Theorem, there is a cutset~$X$ of size less than~$\ell$ separating~$V(S)$ from~$V(T)$.
    Then there exists an element~$B$ of~$\Beta$ disjoint from~$X$ because the order of~$\Beta$ is at least~$\ell$. 
    Since~$G[B]$ is connected, $V(S)$ or~$V(T)$ does not intersect the component of~${G-X}$ containing~$B$. 
    By symmetry, we may assume that~$S$ does not intersect the component of~${G-X}$ containing~$B$. 
    Let~$\Beta'$ be the subbramble of~$\Beta$ consisting of all elements intersecting~$S$. 
    Since all elements of~$\Beta'$  either intersect~$B$ or are joined by an edge to~$B$, they all intersect~$X$. 
    But then~$X$ is a hitting set for~$\Beta'$ and therefore~$\Beta'$ has order at most~${\abs{X} < \ell}$, which is a contradiction. 
\end{proof}

The following theorem due to Seymour~\cite{Seymour1980c} will allow us to specify the planar minor we should consider. 

\begin{theorem}[Seymour~{\cite[4.1]{Seymour1980c}}]
    \label{2DRP}
    Let~$G$ be a graph and let~${T = \{s_1,t_1,s_2,t_2\}}$ be a subset of~$V(G)$. 
    If~$G$ does not contain two vertex-disjoint paths, 
    one linking~$s_1$ to~$t_1$ and the other linking~$s_2$ to~$t_2$, 
    then there is a graph ~$J$ satisfying all of the following. 
    \begin{enumerate}[label=(S\arabic*)]
        \item \label{s1} ${T \subseteq V(J) \subseteq V(G)}$.
        \item \label{s2} For every component~$U$ of~${G-V(J)}$, the set~${S_U := N_G(V(U))}$ has at most three vertices. 
        \item \label{s3} $J$ is obtained from~${G[V(J)]}$ by adding edges minimally so that each~$S_U$ is a clique for every component~$U$ of~${G-V(J)}$.
        \item \label{s4} $J$ has an embedding in the plane with the vertices of~$T$ appearing around some face of the embedding in the order~$s_1$,~$s_2$,~$t_1$,~$t_2$. 
    \end{enumerate}
\end{theorem}

The following easy corollary of this result shows that we can insist that~$J$ is a minor of~$G$.

\begin{corollary}\label{2DRP-cor}
    Let~$G$ be a graph and let~${T = \{s_1,t_1,s_2,t_2\}}$ be a subset of~$V(G)$.
    If~$G$ does not contain two vertex-disjoint paths 
    one linking~$s_1$ to~$t_1$ and the other linking~$s_2$ to~$t_2$, then there is a graph~$J$  satisfying \ref{s1}, \ref{s2}, \ref{s3}, \ref{s4} and furthermore~$J$ is a minor of~$G$. 
\end{corollary}

\begin{proof}
    We proceed by the induction on~$\abs{V(G)}$. 
    Suppose that~$G$ has a minimal cutset~$S$ of size at most two separating some component~$U$ of~${G -S}$ from~$T$. 
    Then let~$G'$ be the graph obtained from~${G-V(U)}$ by adding an edge if necessary to turn~$S$ into a clique. 
    By the induction hypothesis, $G'$ has a minor~$J$ satisfying \ref{s1}, \ref{s2}, \ref{s3}, \ref{s4} for~$G'$. 
    Then it follows that~$J$ satisfies \ref{s1}, \ref{s2}, \ref{s3}, \ref{s4} for~$G$, regardless of the size of~${S \cap V(J)}$. 
    Now observe that~$J$ is a minor of~$G$.
    So, we can assume that no cutset of size at most~$2$ separates a component of~${G-S}$ from~$T$.
    
    We choose a graph~$J$ guaranteed to exist by \cref{2DRP} with maximal~$V(J)$. 
    If~${V(J) = V(G)}$, then~${J = G}$ and we are done. 
    Otherwise, consider any component~$U$ of~${G-V(J)}$. 
    Observe that~${G[V(U) \cup S_U]}$ cannot be a tree because then we could draw it in the plane with~$S_U$ on the infinite face, contradicting the maximality of~$V(J)$. 
    So, $U$ contains a cycle~$C_U$. 
    By the previous paragraph and Menger's theorem, there are three vertex-disjoint paths from~$C_U$ to~$S_U$. 
    It follows that~$J$ is a minor of~$G$. 
\end{proof}

The following theorem will allow us to find the desired structure in that planar minor. Initially, this theorem was proved by Robertson, Seymour, and Thomas~\cite{RST1994} with a slightly worse bound.

\begin{theorem}[Gu and Tamaki~\cite{GT2012}]
    \label{planartw}
    Let~$g$ be a positive integer.
    Every planar graph of treewidth at least~${\frac{9}{2}g-5}$
    has a~${g \times g}$ grid as a minor. 
\end{theorem}

\begin{lemma}
    \label{lem:pathpartition}
    Let~$c_1$,~$c_2$ be positive integers.
    Let~$G$ be a graph and let~$P$ be a path from~$x$ to~$y$. 
    Let~$\Beta$ be a bramble of order at least~${c_1+c_2}$ in~$G$. 
    If~$V(P)$ intersects every element of~$\Beta$, 
    then~$P$ can be partitioned into two edge-disjoint subpaths~$P_1$, $P_2$ such that 
    \begin{enumerate}[label=(\roman*)]
        \item ${x \in V(P_1)}$, ${y \in V(P_2)}$, 
        \item the subbramble~$\Beta_1$ of~$\Beta$ consisting of all elements of~$\Beta$ intersecting~$V(P_1)$ has order exactly~$c_1$, 
        \item the subbramble~$\Beta_1'$ of~$\Beta$ consisting of all elements of~$\Beta$ intersecting $V(P_1)-V(P_2)$ has order at most~$c_1-1$, 
        \item  the subbramble $\Beta-\Beta_1$ of $\Beta$ has order at least $c_2$ and $V(P_2)-V(P_1)$ intersects every element of $\Beta-\Beta_1$, and 
        \item the subbramble~${\Beta-\Beta_1'}$ of~$\Beta$ has order at least~$c_2+1$ and~$V(P_2)$ intersects every element of~${\Beta-\Beta_1'}$.
    \end{enumerate}
\end{lemma}

\begin{proof}
    We choose a minimal subpath~$P_1$ of~$P$ starting at~$x$ such that~$\Beta_1$ has order at least~$c_1$. 
    Then~$\Beta_1$ has order exactly~$c_1$ because otherwise, we could shorten~$P_1$ by removing its last vertex, contradicting its minimality. 
    By the minimality, $\Beta_1'$ has order at most~$c_1-1$.
    Note that there are~$c_1$ vertices intersecting all members of~$\Beta_1$ and therefore the order of~${\Beta-\Beta_1}$ is at least~$c_2$ because the order of~$\Beta$ is at least~${c_1+c_2}$. 
    Similarly, the order of~${\Beta-\Beta_1'}$ is at least~$c_2+1$.
\end{proof}

\begin{lemma}
    \label{twisted-or-planar}
    Let~$\ell_1$, $\ell_2$ be positive integers.
    If a graph~$G$ has a bramble~$\Beta$ of order~${\ell \geq \ell_1 + \ell_2 + 5}$, 
    then at least one of the following holds. 
    \begin{enumerate}[label=(\arabic*)]
        \item \label{itm:twisted} $G$ has two vertex-disjoint cycles~$C_1$, $C_2$ such that for each~${i \in \{1,2\}}$, 
        the subbramble of~$\Beta$ consisting of elements intersecting~$C_i$ has order at least~$\ell_i$. 
        \item \label{itm:planar} $G$ has a planar minor~$J$ having a bramble of order at least~${\ell-\ell_1-\ell_2+2}$. 
    \end{enumerate}
\end{lemma}

\begin{proof}
    Let~$G$ be a graph.
    Let~$\Beta$ be a bramble of~$G$ of order~${\ell \geq \ell_1 + \ell_2 + 5}$, and let~${t := \ell-\ell_1-\ell_2}$. 
    We apply \cref{cyclebramble} to obtain a cycle~$C$ which is a hitting set for~$\Beta$. 

    Let~${f = vw}$ be an edge of~$C$.
    We obtain an edge-partition of~$C$ into four subpaths~$P_1$, $P_3$, $P_2$, $P_4$, 
    in this cyclic order with~${V(P_1) \cap V(P_4) = \{v\}}$ and~${w \in V(P_4)}$, by 
    applying~\cref{lem:pathpartition} several times as follows.

    \begin{enumerate}[label=(\alph*)]
        \item We first apply it to~${C - f}$ and~$\Beta$ to obtain paths~$P_1$ and~${Q_1}$ such that the subbramble~$\Beta_1$ consisting of all elements of~$\Beta$ intersecting~$V(P_1)$ has order exactly~$\ell_1$, 
        the subbramble~$\Beta_1'$ consisting of all elements of~$\Beta$ intersecting~$V(P_1)-V(Q_1)$ has order at most~${\ell_1-1}$,
        and~${\Beta_{3,2,4} := \Beta-\Beta_1'}$ has order at least~${\ell-\ell_1+1}$. 
        \item Secondly, we apply it to~$C[V(Q_1) \cup \{v\}]$ and~$\Beta_{3,2,4}$ to obtain~$Q_{3,2}$ and~$P_4$ such that the subbramble~$\Beta_{3,2}$ consisting of all elements of~$\Beta_{3,2,4}$ intersecting~$V(Q_{3,2})$ has order exactly~${\ell_2 + 3}$, the subbramble~$\Beta'_{3,2}$ consisting of all elements of~$\Beta_{3,2,4}$ intersecting~${V(Q_{3,2}) - V(P_4)}$ has order at most~${\ell_2 + 2}$, and~$\Beta_4 := \Beta_{3,2,4} - \Beta'_{3,2}$ has order at least~${\ell-\ell_1-\ell_2-1 \geq 4}$. 
        \item Lastly, we apply it to~${Q_{3,2}}$ and~$\Beta_{3,2}$ to obtain~$P_2$ and~$P_3$ such that the subbramble~$\Beta_2$ consisting of all elements of~$\Beta_{3,2}$ intersecting~$V(P_2)$ has order exactly~$\ell_2$, the subbramble~$\Beta_2'$ consisting of all elements of~$\Beta_{3,2}$ intersecting~$V(P_2) - V(P_3)$ has order at most~$\ell_2 - 1$, and ${\Beta_3 := \Beta_{3,2} - \Beta_2'}$ has order at least~$4$. 
    \end{enumerate}

    For each~${i \in \{1,2\}}$ let~$s_i$ and~$t_i$ be the endvertices of~$P_i$ labelled such that~$t_i$ is an endvertex of~$P_{i+2}$. 
    Let~$H$ be the induced subgraph of~$G$ obtained by deleting all internal vertices of both~$P_1$ and~$P_2$. 
    If~$H$ has two vertex-disjoint paths~$R_1$,~$R_2$ linking~$s_1$ to~$t_1$ and~$s_2$ to~$t_2$ respectively, 
    then property~\ref{itm:twisted} is witnessed by the two vertex-disjoint cycles~${P_1 \cup R_1}$ and~${P_2 \cup R_2}$ with~$\Beta_1$ and~$\Beta_2$, respectively. 
    
    Thus we may assume that~$H$ does not have two vertex-disjoint paths linking~$s_1$ to~$t_1$ and~$s_2$ to~$t_2$, respectively. 
    By \cref{2DRP-cor} applied to~$H$, there is a planar minor~$J$ of~$H$ satisfying \ref{s1}, \ref{s2}, \ref{s3}, \ref{s4} of \cref{2DRP}. 
    
    Let~${\Beta_{3,4} := \Beta_{3,2,4} - \Beta_2'}$ and~${\Beta^* := \{ B \cap V(J) \mid B \in \Beta_{3,4} \}}$. 
    Observe that since $\Beta_{3,2,4}$ has order at least~${\ell-\ell_1+1}$ and~$\Beta_2'$ has a hitting set of size at most~$\ell_2-1$, the order of~$\Beta_{3,4}$ is at least~$(\ell-\ell_1+1)-(\ell_2-1)=t+2$. 

    \begin{claim}
        $\Beta^*$ is a bramble in~$J$.
    \end{claim}
    
    \begin{subproof}
        First, we show that~${B \cap V(J) \neq \emptyset}$ for all~${B \in \Beta_{3,4}}$. 
        Suppose not. 
        Since ${G[B] = H[B]}$ is connected, there is a component~$U$ of~${H-V(J)}$ that contains~$B$. 
        By \ref{s2}, ${S_U := N_H(V(U))}$ is a clique in~$J$ with at most three vertices. 
        There is some~${i \in \{3,4\}}$ such that~${\abs{V(P_i) \cap S_U} \leq 1}$. 
        By \ref{s1}, both ends of~$P_i$ are in~$J$ and therefore~${P_i \cap V(U) = \emptyset}$. 
        We claim that~$S_U$ is a hitting set for~$\Beta_{i}$. 
        Suppose not. 
        Then there is some~${B' \in \Beta_{i}}$ such that~${B' \cap S_U = \emptyset}$. 
        Since ${B' \cap V(P_i) \neq \emptyset}$, we have~${B' \cap V(U) = \emptyset}$. 
        Since both~$B$ and~$B'$ are in~$\Beta_{3,4}$ and ${B \subseteq V(U)}$, $H$ has an edge joining a vertex of~$B$ to a vertex of~$B'$, contradicting the fact that~$B'$ is disjoint from~$S_U$. 
        Thus, we deduce that~$S_U$ is a hitting set for~$\Beta_{i}$. 
        This is a contradiction because~$\Beta_{i}$ has order at least~$4$. 
    
        We show for all~${B_1,B_2 \in \Beta_{3,4}}$ that both~$J[B_1 \cap V(J)]$ and ${J[(B_1 \cup B_2) \cap V(J)]}$ are connected. 
        Let~${B_1,B_2 \in \Beta_{3,4}}$, not necessarily distinct, and let~${B := B_1 \cup B_2}$. 
        Let~${\mathcal{U} := \{ U \mid \text{$U$ is a component of~${H-V(J)}$ with~${B \cap V(U) \neq \emptyset}$} \}}$. 
        For each \linebreak component~${U \in \mathcal U}$, let~${X_U := E(H[B \cap V(U)]) \cup \{e_U\}}$ where~$e_U$ is an edge of ${H[B]}$ joining~$S_U$ with a vertex in~${B \cap V(U)}$. 
        Since~$S_U$ is a clique in~$J$, we have that~${H[B] / \bigcup_{U \in \mathcal U} X_U}$ is a connected spanning subgraph of~${J[B\cap V(J)]}$. 
        Thus, ${J[B \cap V(J)]}$ is connected. 
    \end{subproof}

    Since any hitting set for~$\Beta^*$ is also a hitting set of~$\Beta_{3,4}$, 
    the order of~$\Beta^*$ is at least~$t+2$. 
\end{proof}

Now it is straightforward to prove \cref{twistedprism1,twistedprism2}.

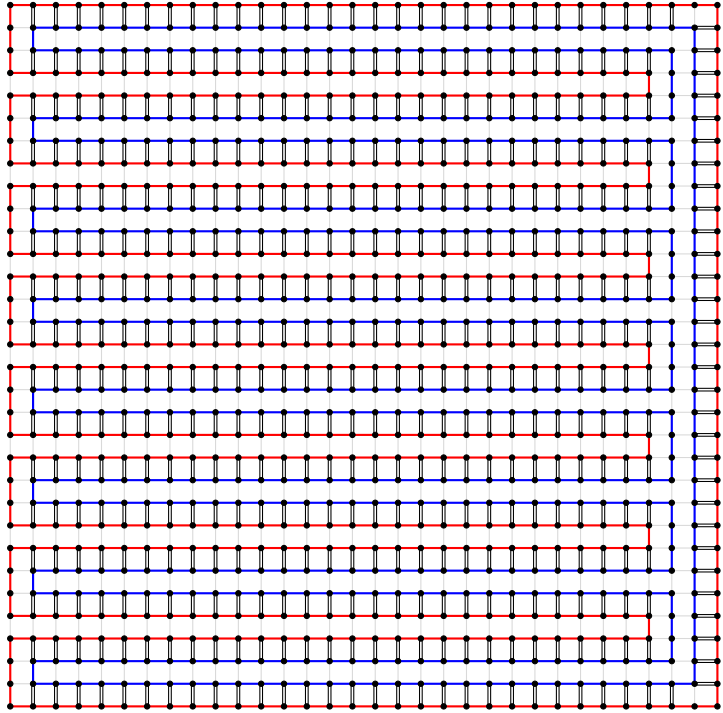
\begin{figure}
    \centering
    \begin{tikzpicture}[scale=.3]
        \tikzstyle{vv}=[circle, draw, solid, fill=black, inner sep=0pt, minimum width=2pt]
        \foreach \x in {0,1,...,31}{
            \draw [color=gray!30] (\x,0)--(\x,31);
            \draw [color=gray!30] (0,\x)--(31,\x);
        }
        \foreach \y in {0,4,8,...,28}{
            \draw [color=red,thick] (28,\y)--(0,\y)--(0,\y+3)--(28,\y+3);
            \draw [color=blue,thick] (29,\y+1)--(1,\y+1)--(1,\y+2)--(29,\y+2);
            \foreach \x in {1,2,...,28} {
                \draw [double] (\x,\y)--(\x,\y+1);
                \draw [double] (\x,\y+2)--(\x,\y+3);
            }
        }
        \foreach \y in {0,4,8,...,24}{
            \draw [color=red,thick] (28,\y+3)--(28,\y+4);
            \draw [color=blue,thick] (29,\y+2)--(29,\y+5);
        }
        \foreach \y in {1,2,...,30}{
            \draw [double](30,\y)--(31,\y);
        }
        \draw [double](29,0)--(29,1);
        \draw [double] (29,30)--(29,31);
        \draw[color=red,thick] (28,0)--(31,0)--(31,31)--(28,31);
        \draw[color=blue,thick] (29,1)--(30,1)--(30,30)--(29,30);
        \foreach \x in {0,1,...,31}
            \foreach \y in {0,1,...,31}
                \node at (\x,\y) [vv] {};
    \end{tikzpicture}
    \caption{A ${4r \times 4r}$ grid contains a ${(8r^2-4r)}$-prism as a minor. Here, ${r = 8}$.}
    \label{fig:twistedprism2}
\end{figure}

\twistedprismone*

\begin{proof}
    Let~${r := \lceil \frac{1+\sqrt{2\ell+1}}{4} \rceil}$. 
    Suppose that a graph~$G$ has treewidth at least ${2\ell + 18r - 7}$ and has no twisted $\ell$-prism as a minor. 
    Let~$\Beta$ be a maximum order bramble of~$G$. 
    By \cref{duality}, the order of~$\Beta$ is at least~${2\ell + 18r - 6}$. 
    By \cref{menger}, we may assume that~$G$ has no two vertex-disjoint cycles such that each of them intersects every member of some subbramble of order at least~$\ell$ of~$\Beta$. 
    By \cref{twisted-or-planar}, $G$ has a planar minor~$J$ 
    having a bramble of order at least~${18r-4}$. 
    By \cref{duality}, the treewidth of~$J$ is at least~${18r-5 = \frac{9}{2}\cdot 4r-5}$. 
    By \cref{planartw}, $J$ has a ${4r \times 4r}$ grid as a minor. 
    Observe from \cref{fig:twistedprism2} that a ${4r \times 4r}$ grid contains a ${(8r^2-4r)}$-prism 
    and~${8r^2-4r = \frac{1}{2}(4r-1)^2 - \frac{1}{2} \geq \ell}$. 
    It follows that~$J$ has a $\ell$-prism as a minor, contradicting our assumption. 
\end{proof}

\twistedprismtwo*

\begin{proof}
    Suppose that a graph~$G$ has treewidth at least~${2 \ell + 11}$ and has no twisted $\ell$-prism as a minor. 
    Let~$\Beta$ be a maximum order bramble of~$G$. 
    By \cref{duality}, the order of~$\Beta$ is at least~${2 \ell + 12}$. 
    By \cref{menger}, we may assume that~$G$ has no two vertex-disjoint cycles such that each of them intersects every member of some subbramble of order at least~$\ell$ of~$\Beta$. 
    By \cref{twisted-or-planar}, $G$ has a planar minor~$J$ having a bramble of order at least~$14$.
    By \cref{duality}, the treewidth of~$J$ is at least~$13$. 
    By \cref{planartw}, $J$ has a ${4 \times 4}$ grid as a minor. 
\end{proof}

Now let us prove \cref{no4by4s}.
To do so, we prove the following lemma, which, combined with \cref{twistedprism2}, implies the theorem immediately.

\begin{lemma}\label{75prism}
    Every twisted $75$-prism contains a ${4 \times 4}$ grid minor.
\end{lemma}

\begin{proof}
    Let~$G$ be a twisted $75$-prism, and let~$C_1$,~$C_2$ be two cycles of $G$ with a matching of size~$75$ between them. 
    Let~${v_1,v_2,\ldots,v_{75}}$ be the vertices of~$C_1$ in the cyclic order, let~${w_1,w_2,\ldots,w_{75}}$ be the vertices of~$C_2$ in the cyclic order, and let~${\pi \colon \{ 1, \dots, 75 \} \to \{ 1, \dots, 75 \}}$ be the permutation for which~${e_i := v_i w_{\pi(i)}}$ is an edge of~$G$ for all~${i \in \{1, \dots, 75\}}$. 

    Suppose that~$G$ has a cycle~$C$ of length~$4$. 
    Without loss of generality, we may assume that~$C$ contains~$e_1$ and~$e_{75}$. 
    Since ${\abs{\{e_2,\dots,e_{74}\}} = 8 \cdot 9 + 1}$, we can apply the Erd\H{o}s-Szekeres theorem~\cite{ES1935} to find 
    \begin{enumerate}[label=(\roman*)]
        \item\label{item:increasing} an increasing sequence of integers~${2 \leq i_1 < i_2 < \cdots < i_9 \leq 74}$ such that ${\pi(i_1) < \pi(i_2) < \cdots < \pi(i_9)}$, or
        \item\label{item:decreasing}  an increasing sequence of integers~${2 \leq i_1 < i_2 < \cdots < i_{10} \leq 74}$ such that ${\pi(i_1) > \pi(i_2) > \cdots > \pi(i_{10})}$. 
    \end{enumerate}
    Using both~$e_1$ and~$e_{75}$ in case~\ref{item:increasing} and using~$e_1$ in case~\ref{item:decreasing}, $G$ contains a planar minor consisting of two vertex disjoint cycles of length~$12$ and a matching of size~$11$ between them. 
    By allowing one of these cycles to form the outer face and contracting some edges of the other cycle, we observe that~$G$ contains the ${4 \times 4}$ grid as a minor. 

    Thus we may assume that~$G$ has no cycle of length~$4$. 
    Without loss of generality, we may assume that~${\pi(75) = 75}$. 
    Since~${74 > 8^2}$, by applying the Erd\H{o}s-Szekeres theorem~\cite{ES1935} we find
    \begin{enumerate}[label=(\roman*)]
        \item\label{item:increasing2} an increasing sequence of integers~${1 \leq i_1 < i_2 < \cdots < i_9 \leq 74}$ such that ${\pi(i_1) < \pi(i_2) < \cdots < \pi(i_9)}$, or
        \item\label{item:decreasing2} an increasing sequence of integers~${1 \leq i_1 < i_2 < \cdots < i_{9} \leq 74}$ such that ${\pi(i_1) > \pi(i_2) > \cdots > \pi(i_{9})}$. 
    \end{enumerate}
    Using~$e_{75}$ as well as~$e_{i_1}, \dots, e_{i_9}$ we obtain that~$G$ contains a subdivision of the $10$-prism~$H'$ in each case. 
    Let~$C_i'$ denote the cycle of~$H'$ corresponding to~$C_i$ in~$G$ for each~${i \in \{1,2\}}$. 
    For an edge~$e$ of~$H'$, let~${\ell(e)}$ be the length of the path in~$G$ corresponding to~$e$. 

    We claim that if~$xyzw$ is a path of length three in~$C_i'$ for some~${i \in \{1,2\}}$, and~${\ell(xy), \ell(zw) > 1}$, then~$G$ has a ${4 \times 4}$ grid as a minor. 
    By symmetry, we may assume that~${i = 1}$. 
    Let~$x'$, $y'$, $z'$, $w'$ be the vertices of~$C_2'$ that are matched to~$x$, $y$, $z$, $w$ by the edges of~$H'$. 
    Then it is easy to observe that~$H'$ contains a minor isomorphic to the ${4 \times 4}$ grid, which can be seen by contracting edges of the subpaths of~$C_2'$ from~$x'$ to~$y'$ and from~$z'$ to~$w'$. 

    Thus, if we assume for contradiction that~$G$ contains no ${4 \times 4}$ grid minor, it is straightforward to observe that at most~$4$ edges of~$C_i'$ are subdivided in~$G$ for each~${i \in \{1,2\}}$. 
    Thus~$G$ has a cycle of length~$4$, contradicting our assumption. 
\end{proof}
    
\nofourbyfours*

\begin{proof}
    Let~$G$ be a graph with treewidth at least~${2 \cdot 75 + 11}$. 
    By \cref{twistedprism2}, $G$ contains a twisted $75$-prism or a ${4 \times 4}$ grid as a minor. 
    By \cref{75prism}, if~$G$ contains a twisted $75$-prism as a minor, then~$G$ contains a ${4 \times 4}$ grid as a minor. 
\end{proof}

\section{Excluding Disjoint Unions of Cycles}\label{sec:fewcycles}

In this section, we prove \cref{fewcycles}. 
Our approach to the~${r=2}$ case is similar to that used in the proof of \cref{twistedprism1}. 
Let us state it as a separate proposition and then prove it. 

\begin{proposition}
    \label{fewcycles2}
    Let~$H$ be the disjoint union of two cycles. 
    Every graph without an~$H$ minor has treewidth less than~${\abs{V(H)} + \frac{9}{2} \lceil \sqrt{4+\abs{V(H)}} \rceil + 2}$. 
\end{proposition}

\begin{proof}%[Proof of \cref{fewcycles}]
    We let~$C_1$ and~$C_2$ be the two cycles whose disjoint union is~$H$. 
    We set~${\ell_i := \abs{E(C_i)}}$ for each $i \in \{1,2\}$ and let~${\ell := \ell_1 + \ell_2}$. 
    Let~$G$ be a graph whose treewidth is at least~${\ell + \frac{9}{2} \lceil \sqrt{4+\ell} \rceil + 2}$. 
    Let~$\Beta$ be a maximum order bramble of~$G$. 
    By \cref{duality}, the order of~$\Beta$ is at least~${\ell + \frac{9}{2} \lceil \sqrt{4+\ell} \rceil + 3}$. 
    Note that if a cycle intersects every element of some bramble of order at least~$m$, its length is at least~$m$.
    Since~$G$ has no~$H$ minor, by \cref{twisted-or-planar}, $G$ has a planar minor~$J$ having a bramble of order at least~${\frac{9}{2} \lceil \sqrt{4+\ell} \rceil + 5}$. 
    By \cref{duality}, the treewidth of~$J$ is at least~${\frac{9}{2} \lceil \sqrt{4+\ell} \rceil + 4}$. 
    Let~${g := 2 + \lceil \sqrt{4+\ell} \rceil}$. 
    By \cref{planartw}, $J$ has a ${g \times g}$ grid as a minor. 
    Since~${(g-2)^2 \geq \ell_1 + \ell_2 + 4}$, we deduce that 
    \[ 
        g \geq \left\lceil \frac{\ell_1}{g} \right\rceil + 1 + \left\lceil \frac{\ell_2}{g} \right\rceil + 1.
    \] 

    Note that if~${a,b > 1}$ are integers, then the~${a \times b}$ grid has a Hamiltonian cycle or its subgraph obtained by deleting one corner vertex has a Hamiltonian cycle. 
    Thus, the ${a \times b}$ grid has a cycle of length at least~${ab - 1}$. 

    The first~${\lceil \frac{\ell_1}{g} \rceil + 1}$ rows of the ${g \times g}$ grid has a cycle of length exceeding~$\ell_1$. 
    The next~${\lceil \frac{\ell_2}{g} \rceil + 1}$ rows of this grid contains a cycle of length exceeding~$\ell_2$. 
    So~$G$ has~$H$ as a minor, which is the desired contradiction. 
\end{proof}

To handle the~${r \geq 3}$ case of \cref{fewcycles},
we need the following famous theorem of Erd\H{o}s and P\'{o}sa. 

\begin{theorem}[Erd\H{o}s and P\'osa~\cite{EP1965}]\label{thm:ErdosPosa}
    There is a constant~$c^*$ such that, for every positive integer~$r$, every graph contains either a set of vertices of size at most~${c^* r \log r}$ which hits every cycle or a packing of~$r$ vertex-disjoint cycles. 
\end{theorem}

\begin{proposition}
    \label{fewcycles3ormore}
    There is an absolute constant~$c$ such that for every positive integer~$r$, 
    if~$H$ is the disjoint union of~$r$ cycles, 
    then every graph without an~$H$ minor has treewidth less than
    \[
        \frac{3\abs{V(H)}}{2} + c r^2 \log r.
    \]
\end{proposition}

\begin{proof}
    It is enough to show that there is a constant~$c$ such that if~$H$ is the disjoint union of~$r$ cycles, every graph~$G$ without an~$H$ minor has treewidth less than 
    \[ 
        \frac{3\abs{V(H)}}{2} + c \sum_{k=1}^r k \log k.
    \]
    Let~$c^*$ be a positive integer that is the constant in \cref{thm:ErdosPosa} and let~${c := 4c^*}$. 
    We proceed by induction on~$r$. 
    
    Let~$H$ be a disjoint union of~$r$ cycles of lengths~${c_1,c_2,\dots, c_r}$, with~${c_i \geq c_{i+1}}$ for all~${i \in \{1,2,\ldots,r-1\}}$. 
    If~${r = 1}$, the result follows from \cref{cyclebramble}. 
    Thus we may assume that~${r > 1}$. 
    Let~$G$ be a graph whose treewidth is at least ${\frac{3\abs{V(H)}}{2} + c \sum_{k=1}^r k \log k}$. 
    Let~$\Beta$ be a maximum order bramble of~$G$. 
    By \cref{duality}, the order of~$\Beta$ is~${\tw(G) + 1}$. 
    
    Let~$C$ by a cycle in~$G$ that is a hitting set for~$\Beta$, as guaranteed by \cref{cyclebramble}. 

    Let~$f \in E(C)$. 
    We apply \cref{lem:pathpartition} repeatedly, we obtain vertex-disjoint consecutive subpaths~$P_1$,~$P$,~$P_2$ of~$C$ in the cyclic order in the following way. 

    \begin{enumerate}[label=(\alph*)]
        \item We apply it to~$C-f$ and~$\Beta$ to obtain a path~$P$ such that the subbramble~$\Beta_1$ of elements of~$\Beta$ intersecting~$P$ has order exactly~${c_1 - 2}$. 
        \item We apply it to~$C - V(P)$ and~$\Beta - \Beta_1$ to obtain a path~$P_1$ adjacent to~$P$ such that the subbramble~$\Beta_2$ of elements of~$\Beta-\Beta_1$ intersecting~$P_1$ has order exactly~${2 \lfloor c^* r \log r \rfloor + 2}$. 
        \item We apply it to~$C-V(P \cup P_1)$ and~$(\Beta - \Beta_1) - \Beta_2$ to obtain a path~$P_2$ adjacent to~$P$ such that the subbramble~$\Beta_3$ consisting of elements of~${(\Beta-\Beta_1)-\Beta_2}$ intersecting~$P_2$ has order exactly~${2 \lfloor c^* r \log r \rfloor + 2}$. 
        \item  \label{itm:fewcycle-d} Moreover, the subbramble~${\Beta_4 := \Beta-\Beta_1-\Beta_2-\Beta_3}$ has order at least \linebreak ${\tw(G) + 1 - c_1 - 4 \lfloor c^* r \log r \rfloor - 2}$ and no element of~$\Beta_4$ contains a vertex of ${P \cup P_1 \cup P_2}$. 
    \end{enumerate}
    
    By \cref{menger}, there is a set~$S$ of at least~${2 \lfloor c^* r \log r \rfloor + 2}$ vertex-disjoint paths from~$V(P_1)$ to~$V(P_2)$ in~${G - V(P)}$. 
    If any path~$Q$ in~$S$ has at most~${\frac{1}{2} c_1 - 2}$ internal vertices, then consider a cycle~$O$ in~${G[V(Q\cup P\cup P_1\cup P_2)]}$ containing~$P$. 
    Since~$O$ has length at least~${\abs{V(P)} + 2 \geq c_1}$, it suffices to embed the graph~$H'$ consisting of~${r-1}$ vertex-disjoint cycles of lengths~${c_2,c_3,\ldots, c_r}$ as a minor in~${G-V(O)}$. 
    By \ref{itm:fewcycle-d} and \cref{duality}, the treewidth of~${G-V(O)}$ is at least ${\tw(G) - c_1 - 4 \lfloor c^* r \log r \rfloor - 2 - (\frac{1}{2} c_1 - 2)}$, which is at least
    \[
        \left( \frac{3}{2} \sum_{i=2}^r c_{i} \right) + c \sum_{k=1}^{r-1} k \log k.
    \]
    Thus, ${G-V(O)}$ has~$H'$ as a minor by the induction hypothesis. 
    This implies that~$G$ has~$H$ as a minor. 

    Hence, we may assume that every path in~$S$ has at least~${c_1/2}$ vertices. 
    It follows that every cycle in~${G' := P_1 \cup P_2 \cup \bigcup S}$ has length at least~$c_1$. 
    Therefore it suffices to find a packing of~$r$ vertex-disjoint cycles in~$G'$. 
    It is easy to see that a spanning tree of~$G'$ can be obtained by deleting a single edge of all but one path in~$S$, so ${\abs{E(G')} - \abs{V(G')} = \abs{S} - 2 = -1 + (\abs{S} - 1)}$, and by construction~$G'$ has maximum degree~$3$. 
    Deleting a vertex decreases the difference between the number of edges and the number of vertices by at most~$2$. 
    Therefore every hitting set for the cycles in~$G'$ has size at least~${(\abs{S} - 1)/2 > \lfloor c^* r \log r \rfloor}$. 
    By \cref{thm:ErdosPosa}, $G'$ has a packing of~$r$ vertex-disjoint cycles, and hence~$G$ has~$H$ as a minor. 
\end{proof}

\cref{fewcycles} follows from \cref{fewcycles2,fewcycles3ormore}.

\section{Excluding Subdivisions of Small Planar Graphs}\label{sec:outerplanar}

In this section, we prove \cref{outerplanar} using an approach similar to that applied to prove \cref{twistedprism1}. 

In place of \cref{2DRP}, we apply a corollary of the following result due to Robertson and Seymour~\cite[(5.3)]{RS1995}). 
For a graph $H$, an \emph{$H$-model} in a graph~$G$ is a collection $(T_v)_{v\in V(H)}$ of vertex-disjoint trees in $G$ such that for every edge $xy$ of $H$, there is an edge joining $T_x$ and $T_y$ in $G$.
We note that a graph contains an $H$-model if and only if it has $H$ as a minor. 

\begin{theorem}[Robertson and Seymour~{\cite[(5.3)]{RS1995}}]
    \label{rDRP}
    Let $d$ be a positive integer and 
    $T$ be a set of at most $2d$ vertices of a graph~$G$.
    Let $(T_v)_{v\in V(K_{4d})}$ be a $K_{4d}$-model in $G$. 
    If $G-Y$ has a component containing at least one $T_v$ and at least one vertex of $T$ for all $Y\subseteq V(G)$ with $\abs{Y}<\abs{T}$,
    then $G$ has a $K_{2d}$-model such that each vertex of~$T$ is contained in a distinct tree of the model.
\end{theorem}
     
\begin{corollary}
    \label{coro}
    Let~$d$ be a positive integer and 
    let~$Z$ be a set of at most~$d$ vertices of a graph~$G$.
    Let~$\Beta$ be a bramble of~$G$ of order at least~$4d$.
    If there does not exist a cutset~$Y$ of size less than~$d$ for which the unique component of~${G-Y}$ 
    containing an element of~$\Beta$ is disjoint from~$Z$, then either 
    \begin{enumerate}[label=(\alph*)]
        \item there is a $K_d$-model such that every vertex of~$Z$ is contained in a distinct element of the model, or 
        \item \label{itm:coro2} there is a minor of~$G$ which contains no $K_{4d}$-minor and which contains a bramble whose order is the same as the order of~$\Beta$. 
    \end{enumerate}
\end{corollary}

\begin{proof}
    Suppose for contradiction that~$G$ is a graph containing a set~$Z$ of vertices and a bramble~$\Beta$ violating the statement of the corollary, and subject to this suppose that~$\abs{V(G)}$ is as small as possible. 
    Then~$G$ has~$K_{4d}$ as a minor because otherwise \ref{itm:coro2} holds. 
    Furthermore, given a $K_{4d}$-model in~$G$, there must be a cutset of size less than~$d$  separating a tree of the model 
    from~$Z$, or by applying \Cref{rDRP} and considering the trees of the resultant model 
    containing the elements of~$Z$ we are done. 

    We consider such a cutset~$Y$ of minimal cardinality. 
    Letting~$U$ be the 
    component of~${G-Y}$ containing a tree of the model and applying \Cref{rDRP} to $G[Y \cup U]$,
    we obtain that the graph obtained from~${G-U}$ by adding edges so~$Y$ is a clique is a minor~$G'$ of~$G$. 
    Furthermore, since we did not completely delete any bramble element, the set consisting of the intersection of each bramble element with this minor is a bramble~$\Beta'$ of the same order in~$G'$. 
    Since the component we deleted was attached to a clique of~$G'$, there is no cutset~$Y'$ in~$G'$ 
    of size less than~$d$ for which the unique component of~${G'-Y'}$ 
    containing an element of our new bramble is disjoint from~$Z$. 
    But now in~$G'$ the set~$Z$ together with~$\Beta'$ must violate the statement of the corollary, contradicting the minimality of~$\abs{V(G)}$. 
\end{proof}

In place of \cref{planartw}, we apply the following result of Demaine and Hajiaghayi~\cite{DH2005}.

\begin{theorem}[Demaine and Hajiaghayi~\cite{DH2005}]
    \label{Klfreetw}
     For every positive integer~$a$, there is an integer~${c_a > 1}$ such that for every positive integer $g$, every graph of treewidth at least~${c_a g}$ has  $K_a$ or the ${g \times g}$ grid as a minor.
\end{theorem}

\begin{lemma} 
    \label{corollary}
    Let~$H$ be a minor of a ${g \times g}$ grid.
    Let $\ell$ be a positive integer.
    If~$G$ is a subdivision of~$H$ obtained by subdividing each edge less than~$\ell$ times, 
    then~$G$ is a minor of a ${\lceil 2\sqrt{\ell}\rceil g
    \times \lceil2 \sqrt{\ell}\rceil g}$ grid.
\end{lemma}

\begin{proof}
    Let~$G$ be a subdivision of~$H$ such that each edge of~$H$ is subdivided less than~$\ell$ times. 
    Let~${r := \lceil \frac{1}{2} \sqrt{2\ell-1} - \frac{1}{2} \rceil}$. 
    Note that the~${2r \times r}$ grid has a Hamiltonian cycle using at least one edge in the first column and at least one edge in the first row. 
    
    We now present a mapping~$\varphi$ from the ${g \times g}$ grid to the ${(2r+1) g \times (2r+1) g}$ grid
    such that if a graph~$J$ is a subgraph of the ${g \times g}$ grid, then its image~$\varphi(J)$ is a subdivision of~$J$ in which each edge is subdivided at least~$\ell$ times. 
    We regard ${\{0,1,2,\ldots,g-1\} \times \{0,1,2,\ldots,g-1\}}$ as the vertex set of the ${g \times g}$ grid. 
    For vertices, we define~${\varphi(i,j) := ((2r+1)g,(2r+1)g)}$. 
    For a vertical edge~${(i,j)(i,j+1)}$ of the ${g \times g}$ grid, 
    we map it to a path~$P$ of length~${2r^2+2r}$ from~${\varphi(i,j)}$ to~${\varphi(i,j+1)}$ 
    where the internal vertices of~$P$ is 
    \[
        \{ (x,y) \mid (2r+1)i \leq  x < (2r+1)i+r,~ (2r+1)j < y < (2r+1)(j+1)\}.
    \]  
    For a horizontal edge~${(i,j)(i+1,j)}$ of the ${g \times g}$ grid, 
    we map it to a path~$P$ of length~${2r^2+2r}$ from~${\varphi(i,j)}$ to~${\varphi(i+1,j)}$ 
    where the set of internal vertices of~$P$ is the union of~${\{ (x,(2r+1)j) \mid (2r+1)i < x < (2r+1)(i+1) \}}$ and ${\{ (x,y) \mid (2r+1)i+r < x < (2r+1)(i+1),~(2r+1)j < y < (2r+1)(j+1) \}}$. \linebreak
    Since the ${(2r)\times r}$ grid has a Hamiltonian cycle using at least one edge in the first column and at least one edge in the first row, such a path exists. 
    Note that such a path~$P$ has~${2r^2+2r}$ vertices
    and~${2r^2+2r \geq \frac{2\ell-1}{2} - \frac{1}{2} = \ell-1}$. 
    
    Now it follows that the ${(2r+1)g \times (2r+1)g}$ grid has a $G$-model because we can map the $H$-model in the ${g \times g}$ grid by~$\varphi$. 
    Note that 
    \[
        {2r+1 \leq \lceil \sqrt{2\ell-1} + 1 \rceil \leq 
        \lceil 2\sqrt{\ell}\rceil}. \qedhere
    \] 
\end{proof}

With these preliminaries out of the way, we turn to the proof of the theorem. 

\outerplanar*

\begin{proof}
    Let~$c_{8r}$ be the constant guaranteed by \cref{Klfreetw} for~${a = 8r}$.
    Robertson, Seymour, and Thomas~\cite{RST1994} showed that there exists an integer $g_{2r}$ such that every planar graph with at most~$2r$ vertices is a minor of the~${g_{2r} \times g_{2r}}$ grid. 
    Let~${b_r := \max\{2r^2,8r,12c_{8r}^2g_{2r}^2\}}$.
    
    We claim that for every positive integer~$n$, we have 
    \(
        \frac{r-1}{2}n + b_r  \geq c_{8r} \lceil 2\sqrt{n} \rceil g_{2r}
    \).
    To prove this, we may assume that 
    \(
        {n < 2c_{8r}\lceil 2\sqrt{n}\rceil g_{2r}}
    \). 
    Then 
    \( 
        {n <2 c_{8r} (3\sqrt{n}) g_{2r}}
    \) and therefore 
    ${2\sqrt{n} \leq 12c_{8r}g_{2r}}$. This implies that~${c_{8r}\lceil 2\sqrt{n}\rceil g_{2r} \leq 12 c_{8r}^2 g_{2r}^2 \leq b_r}$.

    Assume for contradiction that the theorem is false for this choice of~$b_r$, and consider a minimal counterexample~$H$ which is a subdivision of a graph~$H'$ with~${d := \abs{E(H')} \leq r}$.
    We can assume that~$H'$ does not contain an isolated vertex~$x$ as then so does~$H$ and \[{f(H-x) \geq f(H)-1\geq \frac{r+1}{2}\abs{V(H-x)} + b_r}\] contradicting the minimality of~$H$. 
    So~${\abs{V(H')} \leq 2d}$. 
    
    We let~${Q_1, \ldots, Q_d}$ be the paths of~$H$ corresponding to the edges of~$H'$ and let~$\ell_i$ be the length of~$Q_i$.
    We pick this labelling so that $\ell_i\geq \ell_{i+1}$ for all $i\in \{1,2,\dots,d-1\}$.
    
    Let~$G$ be a graph that does not contain~$H$ as a minor 
    such that the treewidth of~$G$ is at least~${\frac{r+1}{2} \abs{V(H)} + b_r}$. 
    Let~$\Beta$ be a maximum order bramble in~$G$.
    By~\cref{duality}, $\Beta$ has order at least~${\frac{r+1}{2} \abs{V(H)} + b_r + 1}$. 
    By~\cref{cyclebramble}, there is a cycle~$C$ intersecting all elements of~$\Beta$.

    Note that 
    \begin{align*}
    \sum_{i=1}^d \left( \sum_{j=i}^d(\ell_j-1) + \left\lfloor \frac{b_r}{d} \right\rfloor \right)
    &\leq \sum_{i=1}^d (\ell_i-1) i  +  b_r\\
    &\leq \frac{1}{d} \left(\sum_{i=1}^d (\ell_i-1)\right)\left( \sum_{i=1}^d i \right)+ b_r \\
    %&\text{by the Chebyshev's sum inequality}\\
    &\leq \frac{\abs{V(H)}}{d}\frac{d(d+1)}{2}+b_r\\
    &\leq \frac{r+1}{2}\abs{V(H)}+b_r, 
    \end{align*}
    where the second inequality follows from Chebyshev's sum inequality. %(see for example~\cite{HardyLP51}). 
    Therefore, 
    by \cref{lem:pathpartition}, there exists a partition of~$C$ into vertex-disjoint subpaths~$R_1$, $\ldots$, $R_{d}$ 
    such that the order of the subbramble~$\Beta_i$ of~$\Beta$ consisting of elements of~$\Beta$ intersecting~$R_i$ and not intersecting $\bigcup_{j=1}^{i-1}R_j$ is at least~${\sum_{j=i}^d(\ell_j-1) + \lfloor \frac{b_r}{d} \rfloor}$ for each~${i \in \{1,2,\ldots,d\}}$. 
    
    Let us fix an orientation of~$C$. 
    For each~${i \in \{1,2,\ldots,d\}}$, let~$P_i$ be path formed by the last~${\ell_i-1}$ vertices of~$R_i$. 
    We let~${P_i' := R_i-V(P_i)}$.
    Let~$Z$ be the set of vertices that are endpoints of~$P_i'$ for some~${i \in \{1,2,\ldots,d\}}$. 
    Let~$\Beta'$ be the subbramble of~$\Beta$ consisting of the elements of~$\Beta$ not intersecting~${\bigcup_{i=1}^d V(P_i)}$. 
    Since~${\sum_{i=1}^d (\ell_i-1) < \abs{V(H)}}$, 
    the order of~$\Beta'$ is at least~${\frac{r-1}{2}\abs{V(H)}+b_r+1}$. 
    
    We note that for every~${i \in \{1,2,\ldots,d\}}$, the order of the subbramble~$\Beta_i'$ of~$\Beta'$ consisting of all elements of~$\Beta'$ intersecting~${V(P_i')}$ is at least~${\lfloor \frac{b_r}{d}\rfloor \geq 2r \geq 2d}$. 
    This is because the order of~$\Beta_i$ is at least~${\sum_{j=i}^d (\ell_j-1)+\lfloor \frac{b_r}{d}\rfloor}$ and
    so the subbramble of~$\Beta_i$ consisting of all elements of~$\Beta_i$ not intersecting~${\bigcup_{j=i}^d V(P_j)}$ has order at least~${\lfloor\frac{b_r}{d}\rfloor}$ and this subbramble is a subbramble of~$\Beta_i'$.

    Let~${G' := G - \bigcup_{i=1}^{d}V(P_i)}$. 
    Now, any set~${Y \subseteq V(G)}$ of size at most~${2d-1}$ intersects some~$P_i'$ in at most one vertex. 
    Since the order of~$\Beta'_i$ is at least~$2d$, there is some~${B \in \Beta'_i}$ disjoint from~$Y$.
    Some endpoint of~$P'_i$ is in the same component~$K$ of~${G' - Y}$ as~$B$, because~$B$ intersects~$V(P_i')$. 
    Since $\Beta'$ is a bramble, $K$ is the unique component of~${G' - Y}$ containing an element of~$\Beta'$.
    Thus $G'$ has no cutset~$Y$ of size less than~$2d$ that  such that the unique component of~${G' - Y}$ containing an element of~$\Beta'$ is disjoint from~$Z$. 
    
    Now, if there is a $K_{2d}$-model in~${G'}$ such that each element of~$Z$ is contained in a distinct tree of the model, we can find~$H$ as a minor of~$G$ as follows.
    We first contract each tree of the model so that~$Z$ becomes a clique. 
    We then embed the internal vertices of each path~$Q_i$ using the path~$P_i$, and contract appropriate subcliques of~$Z$ to obtain the vertices in~${V(H')}$. 

    Note that~$\Beta'$ is a bramble of~$G'$ and 
    the order of~$\Beta'$ is at least~${b_r \geq 8r \geq 8d}$.
    Hence, by \Cref{coro} applied to~$\Beta'$ and~$2d$, we may assume that~$G'$ has a minor~$G''$ such that~$G''$ has treewidth at least~${\frac{r-1}{2}\abs{V(H)}+b_r}$ and~$G''$ does not contain~$K_{8r}$ as a minor. 
    Now, since 
    \(
        \frac{r-1}{2}\abs{V(H)} + b_r \geq c_{8r} \lceil 2\sqrt{\abs{V(H)}}  \rceil g_{2r}
    \),
    \cref{Klfreetw} implies that~$G''$ has a ${(\lceil2\sqrt{\abs{V(H)}} \rceil g_{2r})\times (\lceil2\sqrt{\abs{V(H)}} \rceil g_{2r})}$ grid as a minor.
    By \cref{corollary} and the definition of~$g_{2r}$, it follows that~$G''$ contains~$H$ as a minor, a contradiction. 
\end{proof}

\section{Open problems}\label{sec:open}

\cref{smallcomponents,prop:fewcycles}
suggest the following.
\begin{question}
    Is it true that for every constant~$d$, ${f(H) = O(\abs{V(H)})}$ for the family of graphs~$H$ containing at most~${\frac{d\abs{V(H)}}{\log \abs{V(H)}}}$ vertex-disjoint cycles? 
\end{question}

By \cref{removingtree}, if~$H$ is the disjoint union of~$H_1$ and~$H_2$ where~$H_2$ is a forest, then~${f(H) \leq f(H_1) + f(H_2)}$. 
Just by considering disjoint unions of cycles, we see that there exist graphs~$H_1$ and~$H_2$ such that for the disjoint union~$H$ of~$H_1$ and~$H_2$, we have~${f(H) > f(H_1) + f(H_2)}$. 

It is natural to ask if there is a constant~$c$ such that~${f(H) \leq c(f(H_1)+f(H_2))}$. 
This would follow immediately if the following conjecture were to be proven true: 

\begin{conjecture}
    There is a constant~${\epsilon > 0}$ such that the vertex set of every graph of positive treewidth~$w$ can be partitioned into two sets each inducing a subgraph of treewidth at least~${\lfloor \epsilon w \rfloor}$. 
\end{conjecture}

We ask two further questions about how~$f(H)$ grows with small changes to~$H$. 

\begin{question}
    Given an $n$-vertex planar graph~$H$ and a vertex~${v \in V(H)}$, what is the maximum possible difference between~$f(H)$ and~${f(H-v)}$? 
\end{question}

\begin{question}
    Given an $n$-vertex planar graph~$H$ and an edge~${e \in E(H)}$, what is the maximum possible difference between~$f(H)$ and~${f(H-e)}$? 
\end{question}

\bibliographystyle{amsplain}
\bibliography{treewidth}
\end{document}